\documentclass[11pt]{amsart}

\usepackage{amsmath}
\usepackage{fullpage}
\usepackage{xspace}
\usepackage[psamsfonts]{amssymb}
\usepackage[latin1]{inputenc}
\usepackage{graphicx,color}
\usepackage[curve]{xypic} 
\usepackage{hyperref}
\usepackage{graphicx}

\usepackage{amsmath}%
\usepackage{amsthm}%
\usepackage{amscd}
\usepackage{amsfonts}%
\usepackage{amssymb}%
\usepackage{graphicx}

\usepackage{mathrsfs}

\usepackage{tikz}
\usetikzlibrary{matrix,arrows}

\usepackage{tikz-cd}

\newtheorem{theorem}{Theorem}[section]

\newtheorem{corollary}[theorem]{Corollary}

\newtheorem{lemma}[theorem]{Lemma}

\theoremstyle{remark}

\numberwithin{equation}{section}

\newcommand{\Z}{\mathbb{Z}}

\newcommand{\Q}{\mathbb{Q}}

\newcommand{\T}{\mathbb{T}}

\newcommand{\alg}{\mathrm{alg}}

\makeatletter
\@namedef{subjclassname@2020}{%
  \textup{2020} Mathematics Subject Classification}
\makeatother

  \DeclareFontFamily{U}{wncy}{}
    \DeclareFontShape{U}{wncy}{m}{n}{<->wncyr10}{}
    \DeclareSymbolFont{mcy}{U}{wncy}{m}{n}
    \DeclareMathSymbol{\Sha}{\mathord}{mcy}{"58}

\begin{document}
\title[]{Improved bounds for Szpiro's conjecture}
\author{Hector Pasten}
\address{ Departamento de Matem\'aticas,
Pontificia Universidad Cat\'olica de Chile.
Facultad de Matem\'aticas,
4860 Av.\ Vicu\~na Mackenna,
Macul, RM, Chile}
\email[H. Pasten]{hector.pasten@uc.cl}%

\thanks{H.P. was supported by ANID Fondecyt Regular grant 1230507 from Chile.}
\date{\today}

\begin{abstract} In the direction of Szpiro's conjecture for all elliptic curves $E$ over $\mathbb{Q}$, the current best bound is $\log \Delta \ll N \log N$ where $\Delta$ is the absolute value of the minimal discriminant of $E$ and $N$ is the conductor of $E$ (the implicit constant is absolute). This bound dates back to 2013. We obtain the stronger bound $\log \Delta \ll N\log \log N$ which was previously known only under GRH. In addition, for semistable elliptic curves we prove $\log \Delta \ll N$. Our methods build on the theory developed by the author for approaching the $abc$ conjecture via Shimura curves.
\end{abstract}

\maketitle



\section{Introduction}

We will use Vinogradov's notation $X\ll Y$ to mean that there is an absolute constant $c>0$ independent of all parameters for which $|X|\le c\cdot Y$. The constant $c$ is called the implicit constant. If $c$ depends on some parameter this will be indicated with a subscript.

For an elliptic curve $E$ over $\Q$ we write $\Delta$ for the absolute value of its minimal discriminant and $N$ for its conductor. Although this notation does not reflect the dependence on $E$, we hope that this should lead to no confusion. In addition, $h(E)$ denotes the Faltings height of $E$ over $\Q$ as discussed in \cite{Silverman}, not the stable one. From \cite{Silverman} one knows that
\begin{equation}\label{EqnComparison}
\log \Delta \ll \max\{1,h(E)\}
\end{equation}
where the implicit constant is effective. Szpiro's conjecture \cite{Szpiro} predicts that 
$$
\log \Delta \ll \log N \quad\mbox{ (conjectural).}
$$
This is a deep conjecture and it is known that it implies a form of the $abc$ conjecture. Frey \cite{Frey} proposed a stronger version called the \emph{height conjecture}:
$$
h(E) \ll \log N \quad\mbox{ (conjectural).}
$$
Indeed, the bound \eqref{EqnComparison} shows that this last conjecture implies Szpiro's conjecture. We will only discuss Frey's height conjecture, as this is stronger than Szpiro's conjecture.

Progress on these questions has been rather modest. If we restrict our attention to Frey elliptic curves (in particular, $E$ has full rational $2$-torsion) then the bound
$$
h(E) \ll_\epsilon N^{1/3+\epsilon}
$$
holds for every $\epsilon>0$, where the implicit constant is effective \cite{StewartYu}. Recently \cite{PastenThueMahler} the author obtained 
$$
h(E) \ll_\epsilon  N^{1/2+\epsilon}
$$
for all $\epsilon>0$ when the $j$-invariant is an integer, where the implicit constant is effective.

However, if we impose no restriction on $E$, then the best result we know was obtained in 2013 by R. Murty and the author \cite{MurtyPasten}:
\begin{equation}
h(E) \ll N\log N
\end{equation}
with an effective implicit constant. Under GRH, the author obtained in \cite{PastenShimuraCurves} the conditional estimate
\begin{equation}\label{EqnGRH}
h(E) \ll N\log \log N\quad\mbox{ (under GRH).}
\end{equation}

A factorization $N=DM$ is called \emph{admissible} if $\gcd(D,M)=1$, $D$ is squarefree, and $D$ has an even number of prime factors. We write $\phi(n)$ for Euler's totient function.

\begin{theorem}[Main result] \label{ThmMain} For an elliptic curve $E$ over $\Q$ let $N=DM$ be an admissible factorization. Let $\ell$ be a prime not dividing $N$. Then
$$
h(E) \ll M\phi(D)\log(\ell)
$$ 
where the implicit constant is absolute (independent of the previous data) and effective.
\end{theorem}
Suitable choice of admissible factorization and the prime $\ell$ yields the following consequences:
\begin{corollary}\label{Coro1} For all elliptic curves $E$ over $\Q$ we have 
$$
h(E) \ll N\log\log N
$$
where the implicit constant is absolute and effective.
\end{corollary}
This makes \eqref{EqnGRH} unconditional. Furthermore:

\begin{corollary}\label{Coro2} Let $S$ be a finite set of primes. For all elliptic curves $E$ over $\Q$ which are semistable outside $S$ we have 
$$
h(E) \ll_S N
$$
where the implicit constant is effective and only depends on $S$.
\end{corollary}


\section{Proof of the main result}

The goal of this section is to prove Theorem \ref{ThmMain}. All the bounds discussed in this section are effective. We do not repeat it at every step where implicit constants appear.

Let $S^D(M)$ be the space of weight $2$ quaternionic modular forms for $\Gamma_0^D(M)$. Let $f$ be the newform attached to $E$: via the Jacquet--Langlands correspondence it comes from the standard newform attached to $E$ in $X_0(N)$ via the modularity theorem \cite{Wiles, TaylorWiles, BCDT}. 

Let $\T_{D,M}$ be the Hecke algebra acting on $S^D(M)$ generated by Hecke operators of index coprime to $N=DM$. By Jacquet--Langlands, the set of systems of Hecke eigenvalues $\chi:\T_{D,M}\to \Q^\alg$ comes from the Hecke eigenforms in
$$
V_{D,M}=\bigoplus_{d|M}  S_2(\Gamma_0(Dd))^{\rm new} 
$$
where we use new part of spaces of classical modular forms of weight $2$ for $\Gamma_0(m)$. 

Let $r_{D,M}$ be the number of systems of Hecke eigenvalues in $S^D(M)$, so that $r_{D,M} = \dim V_{D,M}$. From Proposition 7.1 in \cite{PastenShimuraCurves} one gets
\begin{equation}\label{EqnrDM}
r_{D,M} \ll \phi(D)M.
\end{equation}
Let us recall that $\ell$ is a prime not dividing $N$. Let $\mu_{D,M}(f,\ell)$ be the multiplicity of $\alpha=a_\ell(f)\in \Z$ among the eigenvalues of the Hecke operator $T_{\ell}$ acting on $V_{D,M}$. Thus, $\mu_{D,M}(f,\ell)$ is the number of systems of Hecke eigenvalues $\chi$ of $S^D(M)$ satisfying $\chi(T_\ell) = \alpha$. 

\begin{lemma}\label{LemmaMultiplicity} We have
$$
\mu_{D,M}(f,\ell)\ll \frac{\phi(D)M \log \ell}{\log N}.
$$
\end{lemma}
\begin{proof}
By Theorem 1 and Lemma 17 of \cite{Martin} we have 
$$
\dim S_2(\Gamma_0(m))^{\rm new}\ll \phi(m).
$$
Let $\mu_m$ be the multiplicity of $\alpha$ as an eigenvalue of $T_\ell$ acting on $S_2(\Gamma_0(m))^{\rm new}$. 

If $m< N^{1/2}$ we have
$$
\mu_{m} \le \dim S_2(\Gamma_0(m))^{\rm new}\ll N^{1/2}.
$$

Let us now assume $m\ge N^{1/2}$. By Theorem 7 of \cite{MurtySinha}
$$
\mu_m \ll \frac{\log \ell}{\log m} \dim S_2(\Gamma_0(m))^{\rm new} + \frac{m}{(\log m)^2} 
$$
 so, recalling $m\ge N^{1/2}$, we get
$$
\mu_m \ll \frac{\log \ell}{\log N} \phi(m) + \frac{m}{(\log N)^2} \ll \frac{\log \ell}{\log N} \phi(m) 
$$
where the last bound uses $\phi(m)\gg m/\log\log m$ for $m\ge 3$. 

Putting together both cases we get
$$
\begin{aligned}
\mu_{D,M}(f,\ell)&= \sum_{\substack{d|M\\ Dd < N^{1/2}}} \mu_{Dd}+\sum_{\substack{d|M\\ Dd \ge N^{1/2}}} \mu_{Dd} \ll \sigma_0(M)N^{1/2} + \sum_{d|M} \frac{\log \ell}{\log (N)} \phi(Dd) \\
&\ll N^{2/3}+\frac{\phi(D) \log \ell}{\log N}\sum_{d|M} \phi(d) \ll \frac{\phi(D)M \log \ell}{\log N}
\end{aligned}
$$
where $\sigma_0(m)$ is the number of divisors of $m$ and we use the elementary bound $\sigma_0(m)\ll_\epsilon m^{\epsilon}$.
\end{proof}

For a system of Hecke eigenvalues $\chi:\T_{D,M}\to \Q^{\alg}$ we let $[\chi]$ be its class under Galois conjugation. If $\chi$ does not come from $f$ and $\chi_0$ is the system coming from $f$, define 
$$
\eta_{f}([\chi]) = [\Z : \chi_0(\ker(\chi))].
$$ 
This is a positive integer that measures congruences between $f$ and $[\chi]$, see Section 1.6 of \cite{PastenShimuraCurves}.

From Theorem 1.5 and Theorem 1.6 in \cite{PastenShimuraCurves} one has 
$$
h(E) \ll \sum_{[\chi]\ne [\chi_0]} \log \eta_{f}([\chi])
$$
for $N\gg 1$, where the sum is over all Galois conjugation classes of systems of Hecke eigenvalues in $S^D(M)$ that do not come from $f$ (the effective condition $N\gg 1$ is just to ensure that the sum is non-empty).

Let us now bound each $\eta_{f}([\chi])$. Split the set of systems of Hecke eigenvalues $\chi : \T_{D,M}\to \Q^\alg$ not coming from $f$ into two sets: $X$ consisting of those with $\chi(T_\ell)\ne \alpha$ and $Y$ consisting of those with $\chi(T_\ell)=\alpha$.

For $\chi\in X$ we choose $p=\ell$ and $P(x)$ as the minimal polynomial of $\chi(T_\ell)$.  Using Proposition 5.4 of \cite{PastenShimuraCurves} we get that $\eta_{f}([\chi])$ divides $P(\alpha)$. We have $P(\alpha)\ne 0$ as $P(x)$ is irreducible and $\chi\in X$, therefore $\eta_{f}([\chi])\le  |P(\alpha)|$. It follows that
$$
\log \eta_{f}([\chi]) \le \log |P(\alpha)|\le \#[\chi]\cdot \log(4\sqrt{\ell})\ll  \#[\chi]\cdot \log(\ell)
$$
using the standard Hasse--Weil bound on Fourier coefficients of Hecke eigenforms.

For $\chi\in Y$ we use the bound obtained in the proof of Theorem 7.2 of \cite{PastenShimuraCurves}:
$$
\log \eta_{f}([\chi]) < \#[\chi]\cdot\left(\log N + \frac{4\log N}{\log\log N}\right) \ll \#[\chi]\cdot \log N.
$$
Putting all together and using \eqref{EqnrDM} together with Lemma \ref{LemmaMultiplicity}, we get
$$
\begin{aligned}
h(E) &\ll \sum_{[\chi]\in X} \log \eta_{f}([\chi]) + \sum_{[\chi]\in Y} \log \eta_{f}([\chi])  \ll \sum_{[\chi]\in X} \#[\chi]\cdot \log(\ell) + \sum_{[\chi]\in Y} \#[\chi]\cdot \log N\\
& \le r_{D,M} \log \ell + \mu_{D,M}(f,\ell) \log N \ll \phi(D) M \log \ell.
\end{aligned}
$$
This concludes the proof of Theorem \ref{ThmMain}. \qed


\section{Proof of the corollaries}

Let $\ell$ be the smallest prime not dividing $N$. Elementary estimates show that $\ell\ll \log N$ and Corollary \ref{Coro1} follows from Theorem \ref{ThmMain} with $D=1$ and $M=N$.

For Corollary \ref{Coro2} we may enlarge $S$ to contain $2$ and all the primes up to the largest element of the original $S$ (if non-empty). We can assume that $N$ is not completely supported on $S$, because the set of elliptic curves with bad reduction contained in $S$ is effectively computable. Let $m$ be the part of $N$ supported on $S$. Let $M=m$ if  $N/m$ has an even number of prime factors, otherwise let $M=qm$ where $q$ is the largest prime divisor of $N$. Let $D=N/M$. Then 
$$
\phi(D)\log\ell = D\log(\ell)\prod_{p|D} \left(1-1/p\right)
$$
Note that every prime $p<\ell$ not dividing $M$ appears in the product, by minimality of $\ell$. Hence, using Mertens's theorem we get
$$
\phi(D)\log\ell \le D\log(\ell)\prod_{p|M}\left(1-1/p\right)^{-1}\prod_{p<\ell} \left(1-1/p\right) \ll \frac{D\log \ell}{\log \ell}\prod_{p|m}\left(1-1/p\right)^{-1}  \ll_S D
$$
where we used $M=qm$ or $M=m$. Corollary \ref{Coro2} now follows from Theorem \ref{ThmMain}.


\section{Acknowledgments}

This research was supported by ANID Fondecyt Regular grant 1230507 from Chile. 

AI models were used as follows. The models suggested an argument leading to a bound of the form $h(E) \ll N(\log \log N)^2$ which was subsequently verified and reconstructed by the author. That motivated further work by the author, leading to the stronger results proved in this paper, including $h(E) \ll N \log \log N$ in general and $h(E) \ll_S N$ for $E$ semistable away from $S$. In addition, AI models were used for proofreading and suggesting presentation improvements. The author takes full responsibility for the content of this paper.

The AI models used (in the period July-September 2026) were OpenAI's ChatGPT-5.6 Sol Pro and ChatGPT-6 Astra, as well as Anthropic's Claude Opus 5 and Claude Fable 5.1.


\end{document}